\documentclass[11pt]{article}

\usepackage[margin=0.92in]{geometry}

\usepackage{amsmath}
\usepackage{amsthm}
\usepackage{amsfonts}
\usepackage{multicol}
\usepackage[numbers,sort&compress]{natbib}
\usepackage{booktabs}
\usepackage{hyperref}

\usepackage{tikz}
\tikzstyle{vertex}=[circle, draw, inner sep=0pt, minimum size=4pt,fill=black]
\newcommand{\vertex}{\node[vertex]}
\usetikzlibrary{decorations.pathreplacing}
\usepackage{subcaption}
\usepackage{enumitem}
\setlist[enumerate]{label={\upshape(\alph*)}}

\usepackage{color}

\newtheorem{theorem}{Theorem}[section]
\newtheorem{lemma}[theorem]{Lemma}

\newtheorem{conjecture}[theorem]{Conjecture}

\theoremstyle{definition}

\begin{document}
	
\title{On the Mean Order of Connected Induced Subgraphs of Block Graphs}
	
\author{Kristaps J.~Balodis,  Lucas Mol, and Ortrud R.~Oellermann\thanks{Supported by an NSERC Discovery Grant, number RGPIN-2016-05237.}\\
University of Winnipeg, Winnipeg, MB, Canada R3B 2E9\\
\small \href{mailto:balodis-k@webmail.uwinnipeg.ca}{balodis-k@webmail.uwinnipeg.ca}, \href{mailto:l.mol@uwinnipeg.ca}{l.mol@uwinnipeg.ca}, \href{mailto:o.oellermann@uwinnipeg.ca}{o.oellermann@uwinnipeg.ca}\\
 Matthew E.~Kroeker\thanks{Supported by an NSERC USRA, number 510336-2017.}\\
 University of Waterloo, Waterloo, ON, Canada N2L 3G1\\
 \small \href{mailto:mattkroeker@shaw.ca}{mattkroeker@shaw.ca}}
\date{November 12, 2018}
\maketitle

\begin{abstract}
The average order of the connected induced subgraphs of a graph $G$ is called the {\em mean connected induced subgraph (CIS) order} of $G$.  This is an extension of the mean subtree order of a tree, first studied by Jamison.  In this article, we demonstrate that among all connected block graphs of order $n$, the path $P_n$ has minimum mean CIS order.  This extends a result of Jamison  from trees to connected block graphs, and supports the conjecture of Kroeker, Mol, and Oellermann that $P_n$ has minimum mean CIS order among all connected graphs of order $n$.
\end{abstract}
	
\section{Introduction}

Jamison~\cite{Jamison1983} initiated the study of the mean subtree order of a tree.  A number of extensions of this mean to other (connected) graphs have recently been considered:
\begin{itemize}
\item the mean order of the sub-$k$-trees of a $k$-tree~\cite{StephensOellermann2018},
\item the mean order of the subtrees (i.e., minimally connected subgraphs) of a graph~\cite{ChinGordonMacpheeVincent2018}, and
\item the mean order of the connected induced subgraphs of a graph~\cite{KroekerMolOellermann2018}.
\end{itemize}
For a tree $T$, all of these means equal the mean subtree order of $T$. However, for connected graphs in general, the last two means have rather different behaviour.

In this article, we continue the study of the average order of the connected induced subgraphs of a graph $G$, called the {\em mean connected induced subgraph (CIS) order} of $G$.  An in-depth study of the mean CIS order of cographs was undertaken in~\cite{KroekerMolOellermann2018}, where the connected cographs of order $n$ having largest and smallest mean CIS order were determined (both the maximum and minimum values tend to $n/2$ asymptotically).  Here, we focus on the mean CIS order of block graphs, i.e., graphs for which every block is complete.  We extend several of Jamison's results~\cite{Jamison1983} on the mean subtree order of trees to the more general setting of the mean CIS order of connected block graphs (note that every tree is a connected block graph).

In particular, Jamison~\cite{Jamison1983} demonstrated that among all trees of order $n$, the path $P_n$ has minimum mean subtree order (or equivalently, mean CIS order).  Our main result is that the path $P_n$ has minimum mean CIS order among all connected block graphs of order $n$.  This supports the conjecture of Kroeker, Mol, and Oellermann~\cite{KroekerMolOellermann2018} that the path $P_n$ has minimum mean CIS order among all connected graphs of order $n$, which has been verified for all $n\leq 9$.

A key tool in the proof of our main result is an extension of the ``local-global mean inequality'' proven by Jamison~\cite{Jamison1983} for trees.  For a given tree $T$, and every vertex $v$ of $T$, Jamison demonstrated that the mean order of all connected induced subgraphs of $T$ containing $v$ (i.e., the ``local'' mean CIS order of $T$ at $v$) is at least as large as the mean CIS order of $T$ (i.e., the ``global'' mean CIS order of $T$). It is known that this inequality between local and global mean CIS orders does not extend to all connected graphs (at least not at every vertex)~\cite{KroekerMolOellermann2018}.  However, we note that it was recently proven, in a more general context, that every graph with nonempty edge set contains at least one vertex at which the local mean CIS order is larger than the global mean CIS order (apply~\cite[Theorem 3.1]{AMW} to the collection of vertex sets that induce connected subgraphs of $G$).  In other words, while the local-global mean inequality does not necessarily hold at \emph{every} vertex of a connected graph $G$, it must hold at \emph{some} vertex of $G$.  In this article, we demonstrate that the local-global mean inequality does hold at every vertex of a connected block graph.  This fact is essential to the proofs of the three key lemmas used to establish our main result.  

We now give a brief description of the layout of the article.  In Section~\ref{Preliminaries}, we provide notation and preliminaries that will be used throughout the article.  In Section~\ref{Main_Result}, we state three key lemmas (the Vertex Gluing Lemma, the Edge Gluing Lemma, and the Stetching Lemma), and we use them to prove our main result.  We then describe an interesting connection between the mean CIS order of block graphs and the mean sub-$k$-tree order of $k$-trees, and explain the implications of our main result in this setting.  In Section~\ref{LocalGlobalSection}, we prove the local-global mean inequality for the mean CIS order of block graphs.  In Section~\ref{KeyLemmas}, we prove the Vertex Gluing Lemma, the Edge Gluing Lemma, and the Stretching Lemma.  We conclude with some open problems.
	
\section{Notation and Preliminaries}\label{Preliminaries}

For a graph $G$, the vertex and edges sets of $G$ are denoted by $V(G)$ and  $E(G)$, respectively.  The \emph{order} of $G$ is $|V(G)|$ and the \emph{size} of $G$ is $|E(G)|$.  For $U\subseteq V(G)$, the subgraph of $G$ induced by $U$ is denoted $G[U]$.  The (open) neighbourhood of  a vertex $v$ of $G$ is denoted $N_G(v)$.

Let $G$ be a graph of order $n$.  Let $\mathcal{C}_G$ denote the collection of connected induced subgraphs of $G$. The \emph{CIS polynomial of $G$} is given by
\[
\Phi_G(x)=\sum_{H\in \mathcal{C}_G}x^{|V(H)|}=\sum_{i=1}^na_ix^i,
\]
where $a_i$ is the number of connected induced subgraphs of $G$ of order $i$ for each $i\in\{1,\dots,n\}$.
One easily verifies that $\Phi_G(1)$ is the total number of connected induced subgraphs of $G$, and that $\Phi'_G(1)$ is the sum of the orders of all connected induced subgraphs of $G$.  Throughout, we use the shorthand notation $N_G=\Phi_G(1)$ and $W_G=\Phi'_G(1)$.  The \emph{mean CIS order of $G$}, denoted $M_G$, is given by
\[
M_G=\frac{\Phi'_G(1)}{\Phi_G(1)}=\frac{W_G}{N_G}.
\]
For a vertex $v\in V(G)$, let $\mathcal{C}_{G,v}$ denote the collection of connected induced subgraphs of $G$ containing $v$.  The \textit{local CIS polynomial of $G$ at $v$} is given by
\[
\Phi_{G, v}(x)=\sum_{H\in\mathcal{C}_{G,v}}x^{|V(H)|}=\sum_{i=1}^nb_ix^i,
\]
where $b_i$ is the number of connected induced subgraphs of $G$ of order $i$ containing $v$ for each $i\in\{1,\dots,n\}.$  So $N_{G,v}=\Phi_{G,v}(1)$ denotes the total number of connected induced subgraphs of $G$ containing $v$, and  $W_{G,v}=\Phi'_{G,v}(1)$ denotes the sum of the orders of all connected induced subgraphs of $G$ containing $v$.  The \emph{local mean CIS order of $G$}, denoted $M_{G,v}$, is given by
\[
M_{G,v}=\frac{\Phi'_{G,v}(1)}{\Phi_{G,v}(1)}=\frac{W_{G,v}}{N_{G,v}}.
\]

The next lemma gives a recursion for the local mean CIS order of a block graph at a cut vertex $v$.  It holds trivially if $v$ is not a cut vertex.

\begin{lemma}\label{LocalSum}
Let $G$ be a block graph with vertex $v$, and let $H_1,\dots, H_k$ be the components of $G-v$.  For $i\in\{1,\dots,k\},$ let $G_i=G[V(H_i)\cup \{v\}]$.   Then
\[
M_{G,v}=\left[\sum_{i=1}^k M_{G_i,v}\right]-(k-1).
\]
Further, we have
\[
M_{G,v}\geq M_{G_i,v}
\]
for all $i\in\{1,\dots,k\}.$
\end{lemma}

\begin{proof}
By a straightforward counting argument,
\[
\Phi_{G,v}(x)=\tfrac{1}{x^{k-1}}\prod_{i=1}^k\Phi_{G_i,v}(x).
\]
Taking the natural logarithm on both sides and differentiating with respect to $x$, we obtain
\[
\frac{\Phi'_{G,v}(x)}{\Phi_{G,v}(x)}=\left[\sum_{i=1}^k\frac{\Phi'_{G_i,v}(x)}{\Phi_{G_i,v}(x)}\right]-\frac{k-1}{x}.
\]
Substituting $x=1$ yields
\[
M_{G,v}=\left[\sum_{i=1}^k M_{G_i,v}\right]-(k-1).
\]
Since $M_{G_i,v}\geq 1$ for all $i\in\{1,\dots,k\},$ it follows that $M_{G,v}\geq M_{G_i,v}$ for all $i\in\{1,\dots,k\}$.
\end{proof}

We extend the notion of the local mean CIS order of a graph $G$ in two natural ways.  For a subset $U$ of $V(G)$, we let $M_{G,U}$ denote the mean order of all connected induced subgraphs of $G$ containing every vertex of $U$.  We let $\Phi_{G,U}(x)$ denote the corresponding generating polynomial.  We let $M^*_{G,U}$ denote the mean order of all connected induced subgraphs of $G$ containing at least one vertex of $U$.  We let $\Phi^*_{G,U}(x)$ denote the corresponding generating polynomial, and $N^*_{G,U}=\Phi^*_{G,U}(1)$ and $W^*_{G,U}(x)=\Phi^{*'}_{G,U}(1).$  Note that if $U$ contains only a single vertex $u$, then $M_{G,U}=M^*_{G,U}=M_{G,u}.$

\begin{lemma}\label{LocalBlockLemma}
Let $G$ be a block graph, and let $U=\{u_1,\dots,u_k\}$ be the vertex set of a single block $B$ of $G$.  For each $i\in\{1,\dots,k\},$ let $G_i$ be the connected component of $G-E(B)$ containing $u_i$.  Then
\begin{align}\label{LocalAtBlock}
M^*_{G,U}=\frac{N^*_{G,U}+1}{N^*_{G,U}}\sum_{i=1}^k \frac{W_{G_i,u_i}}{N_{G_i,u_i}+1}.
\end{align}
\end{lemma}

\begin{proof}
By a straightforward counting argument,
\[
\Phi^*_{G,U}(x)+1=\prod_{i=1}^k \left[ 1+\Phi_{G_i,u_i}(x)\right].
\]
Taking the natural logarithm on both sides and differentiating, we obtain
\[
\frac{\Phi^{*'}_{G,U}(x)}{\Phi^*_{G,U}(x)+1}=\sum_{i=1}^k\frac{\Phi'_{G_i,u_i}(x)}{\Phi_{G_i,u_i}(x)+1}.
\]
Substituting $x=1$ gives
\[
\frac{W^*_{G,U}}{N^*_{G,U}+1}=\sum_{i=1}^k\frac{W_{G_i,u_i}}{N_{G_i,u_i}+1}.
\]
Multiplying both sides by $\frac{N^*_{G,U}+1}{N^*_{G,U}}$ and noting that $M^*_{G,U}=\frac{W^*_{G,U}}{N^*_{G,U}},$ we obtain~(\ref{LocalAtBlock}).
\end{proof}

A key idea that we use in many of our arguments states that if the connected induced subgraphs of a graph $G$ can be partitioned into two or more sets, then $M_G$ is a convex combination (or weighted average) of the mean orders of each of the sets in the partition. This tool was used by Jamison~\cite[Lemma 3.8]{Jamison1983} for the mean subtree order of a tree.   For example, since every connected induced subgraph of $G$ either contains a given vertex $v$, or does not contain $v$, we can write
\[
\Phi_G(x)=\Phi_{G,v}(x)+\Phi_{G-v}(x).
\]
It follows that $M_G$ is a convex combination of $M_{G,v}$ and $M_{G-v}$.  Another useful application of this principle is to disconnected graphs.  If $G$ is a disconnected graph with components $G_1,\dots,G_k$, then $M_G$ is a convex combination of $M_{G_1},\dots,M_{G_k}$.  It follows that $\min\{M_{G_i}\}\leq M_G\leq \max\{M_{G_i}\}.$

\section{Proof of the Main Result}\label{Main_Result}

In this section, we show that among all block graphs of order $n$, the path has minimum mean CIS order.  We use three key lemmas, namely the Vertex Gluing Lemma, the Edge Gluing Lemma, and the Stretching Lemma, which are proven in Section~\ref{KeyLemmas}.  We state these lemmas here, and provide illustrations depicting how they are used in the proof of the main result (see Figure~\ref{LemmaPictures}). By \emph{gluing} two vertices from disjoint graphs, we mean the process of identifying these two vertices.


\medskip

\noindent{\bf The Vertex Gluing Lemma} (Lemma \ref{vertex_gluing})\\
\textit{Let $H$ be a connected block graph of order at least $2$ having vertex $v$.  Fix a natural number $n\geq 3$.  Let $P:u_1\dots u_n$ be a path of order $n$. For $s \in \{1, \dots, n\}$, let $G_s$ be the block graph obtained from the disjoint union of $P_n$ and $H$ by gluing $v$ to $u_s$.  If $1\leq i<j\leq \tfrac{n+1}{2}$, then $M_{G_i}<M_{G_j}.$}
\medskip

In the notation of the Vertex Gluing Lemma, $G_s \cong G_{n-s+1}$.  Thus, it follows from the Vertex Gluing Lemma that $G_1\cong G_n$ has strictly smaller mean CIS order than $G_s$ for all $2\leq s\leq n-1$ (see Figure~\ref{VertexPicture}).  We note that the Vertex Gluing Lemma extends the Gluing Lemma of~\cite{MolOellermann2017} from trees to block graphs.

To simplify the statement of the Edge Gluing Lemma, we will refer to a leaf as a vertex of degree at most $1$ (so that the single vertex of the path of order $1$ is considered a leaf).

 \medskip

\noindent{\bf The Edge Gluing Lemma} (Lemma \ref{edge_gluing})\\
\textit{Let $H$ be a connected block graph of order at least $3$ with adjacent non-cut vertices $u$ and $v$.  Fix a natural number $n\geq 4$.  For $s\in\{1,2,\dots,n-1\}$, let $G_s$ be the graph obtained from $H\cup P_s\cup P_{n-s}$ by gluing a leaf of $P_s$  to $u$ and a leaf of $P_{n-s}$ to $v$.  If $1\leq i<j\leq \tfrac{n}{2}$, then $M_{G_i}<M_{G_j}.$}

\medskip

In the notation of the Edge Gluing Lemma, $G_s \cong G_{n-s}$. Thus, it follows from the Edge Gluing Lemma that $G_1\cong G_{n-1}$ has strictly smaller mean CIS order than $G_s$ for all $2\leq s\leq n-2$ (see Figure~\ref{EdgePicture}).

We now introduce notation used in the statement of the Stretching Lemma. For a fixed integer $n\geq 3$ and $s\in\{1,2,\dots,n-1\}$, let $F_{s,n-s}$ (or $F_s$ for short) denote the graph obtained from the disjoint union of $K_{s}$ and $P_{n-s}$ by joining a leaf of $P_{n-s}$ to every vertex of $K_s$.  Note that if $s=1$, then $F_s\cong P_n$, while if $s=n-1$, then $F_s\cong K_n$.

\medskip

\noindent{\bf The Stretching Lemma} (Lemma \ref{stretching_lemma})
\textit{Let $H$ be a connected block graph of order at least $2$ having vertex $u$.  Fix a natural number $n\geq 3$.  Let $v$ be a vertex of $F_s=F_{s,n-s}$ belonging to the initial $K_s$.  For $s \in \{1, \dots, n-1\}$, let $G_s$ be the block graph obtained from the disjoint union of $H$ and $F_{s,n-s}$ by identifying $u$ and $v$.  If $1\leq i<j\leq n-1$, then $M_{G_i}<M_{G_j}.$}

\medskip

In particular, in the notation of the Stretching Lemma, we see that $G_1$ has strictly smaller mean CIS order than $G_s$ for all $2\leq s\leq n-1$ (see Figure~\ref{StretchingPicture}).

\begin{figure}
\centering{
\begin{subfigure}[b]{\textwidth}
\centering{
\begin{tikzpicture}
\vertex (0) at (0,0) {};
\vertex (1) at (1,0) {};
\vertex (2) at (2.5,0) {};
\vertex (3) at (4,0) {};
\vertex (4) at (5,0) {};
\path
(0) edge (1)
(3) edge (4);
\draw (1,0) -- (1.4,0);
\draw[dotted] (1.4,0)--(2.1,0);
\draw (2.1,0) -- (2.5,0);
\draw (2.5,0) -- (2.9,0);
\draw[dotted] (2.9,0)--(3.6,0);
\draw (3.6,0) -- (4,0);
\draw (2.5,-1) circle [radius=1];
\node[above,yshift=2pt] at (0,0) {\footnotesize $u_1$};
\node[above,yshift=2pt] at (1,0) {\footnotesize $u_2$};
\node[above,yshift=2pt] at (2.5,0) {\footnotesize $u_s$};
\node[above,yshift=2pt] at (4,0) {\footnotesize $u_{n-1}$};
\node[above,yshift=2pt] at (5,0) {\footnotesize $u_n$};
\node[below,yshift=-2pt] at (2.5,0) {\footnotesize $v$};
\node at (2.5,-1) {$H$};
\end{tikzpicture}
\hspace{1cm}
\begin{tikzpicture}
\vertex (3) at (3,0) {};
\vertex (4) at (4,0) {};
\vertex (5) at (5.5,0) {};
\vertex (6) at (6.5,0) {};
\path
(3) edge (4)
(5) edge (6);
\draw (4,0) -- (4.4,0);
\draw[dotted] (4.4,0)--(5.1,0);
\draw (5.1,0) -- (5.5,0);
\draw (3,-1) circle [radius=1];
\node[above,yshift=2pt] at (3,0) {\footnotesize $u_1$};
\node[above,yshift=2pt] at (4,0) {\footnotesize $u_{2}$};
\node[above,yshift=2pt] at (5.5,0) {\footnotesize $u_{n-1}$};
\node[above,yshift=2pt] at (6.5,0) {\footnotesize $u_n$};
\node[below,yshift=-2pt] at (3,0) {\footnotesize $v$};
\node at (3,-1) {$H$};
\end{tikzpicture}
\caption{The Vertex Gluing Lemma: The graph $G_1$ (right) has smaller mean CIS order than the graph $G_s$ (left) for all $2\leq s\leq n-1$.}\label{VertexPicture}}
\end{subfigure}\\
\bigskip
\begin{subfigure}[b]{\textwidth}
\centering{
\begin{tikzpicture}
\vertex (0) at (0,0) {};
\vertex (1) at (1,0) {};
\vertex (2) at (2.5,0) {};
\vertex (3) at (3.5,0) {};
\vertex (4) at (5,0) {};
\vertex (5) at (6,0) {};
\path
(0) edge (1)
(2) edge (3)
(4) edge (5);
\draw (1,0) -- (1.4,0);
\draw[dotted] (1.4,0)--(2.1,0);
\draw (2.1,0) -- (2.5,0);
\draw (3.5,0) -- (3.9,0);
\draw[dotted] (3.9,0)--(4.6,0);
\draw (4.6,0) -- (5,0);
\draw (3,-0.85) circle [radius=1];
\draw [decorate,decoration={brace,amplitude=4pt},xshift=0pt,yshift=6pt]
(-0.2,0) -- (2.7,0) node [above,black,midway,yshift=4pt]
{\footnotesize $s$ vertices};
\draw [decorate,decoration={brace,amplitude=4pt},xshift=0pt,yshift=6pt]
(3.3,0) -- (6.2,0) node [above,black,midway,yshift=4pt]
{\footnotesize $n-s$ vertices};
\node[below,yshift=-2pt] at (2.5,0) {\footnotesize $u$};
\node[below,yshift=-2pt] at (3.5,0) {\footnotesize $v$};
\node at (3,-0.85) {$H$};
\end{tikzpicture}
\hspace{1cm}
\begin{tikzpicture}
\vertex (3) at (3,0) {};
\vertex (4) at (4,0) {};
\vertex (5) at (5.5,0) {};
\vertex (6) at (6.5,0) {};
\path
(3) edge (4)
(5) edge (6);
\draw (4,0) -- (4.4,0);
\draw[dotted] (4.4,0)--(5.1,0);
\draw (5.1,0) -- (5.5,0);
\draw (3.5,-0.85) circle [radius=1];
\draw [decorate,decoration={brace,amplitude=4pt},xshift=0pt,yshift=6pt]
(2.8,0) -- (6.7,0) node [above,black,midway,yshift=4pt]
{\footnotesize $n$ vertices};
\node[below,yshift=-2pt] at (3,0) {\footnotesize $u$};
\node[below,yshift=-2pt] at (4,0) {\footnotesize $v$};
\node at (3.5,-0.85) {$H$};
\end{tikzpicture}
\caption{The Edge Gluing Lemma: The graph $G_1$ (right) has smaller mean CIS order than the graph $G_s$ (left) for all $2\leq s\leq n-2$. Note: $u$ and $v$ must be adjacent in $H$, and must not be cut vertices of $H$.}\label{EdgePicture}}
\end{subfigure}\\
\bigskip
\begin{subfigure}[b]{\textwidth}
\centering{
\begin{tikzpicture}
\draw (-0.5,0) circle[radius=1];
\node at (-0.5,0) {$H$};
\draw (1.25,0) circle[radius=0.75];
\node at (1.25,0) {\footnotesize $K_{s+1}$};
\draw [decorate,decoration={brace,amplitude=4pt},xshift=0pt,yshift=6pt]
(1.8,0) -- (5.7,0) node [above,black,midway,yshift=4pt]
{\footnotesize $n-s$ vertices};
\vertex (1) at (0.5,0) {};
\vertex (2) at (2,0) {};
\vertex (3) at (3,0) {};
\vertex (4) at (4.5,0) {};
\vertex (5) at (5.5,0) {};
\path
(2) edge (3)
(4) edge (5);
\draw (3,0) -- (3.4,0);
\draw[dotted] (3.4,0)--(4.1,0);
\draw (4.1,0) -- (4.5,0);
\end{tikzpicture}
\hspace{1cm}
\begin{tikzpicture}
\draw (1,0) circle[radius=1];
\node at (1,0) {$H$};
\draw [decorate,decoration={brace,amplitude=4pt},xshift=0pt,yshift=6pt]
(1.8,0) -- (5.7,0) node [above,black,midway,yshift=4pt]
{\footnotesize $n$ vertices};
\vertex (2) at (2,0) {};
\vertex (3) at (3,0) {};
\vertex (4) at (4.5,0) {};
\vertex (5) at (5.5,0) {};
\path
(2) edge (3)
(4) edge (5);
\draw (3,0) -- (3.4,0);
\draw[dotted] (3.4,0)--(4.1,0);
\draw (4.1,0) -- (4.5,0);
\end{tikzpicture}
\caption{The Stretching Lemma: The graph $G_1$ (right) has smaller mean CIS order than the graph $G_s$ (left) for all $2\leq s\leq n-1$.}\label{StretchingPicture}}
\end{subfigure}
}
\caption{An illustration of the use of the Vertex Gluing Lemma, the Edge Gluing Lemma, and the Stretching Lemma, where $H$ is a block graph.}
\label{LemmaPictures}
\end{figure}
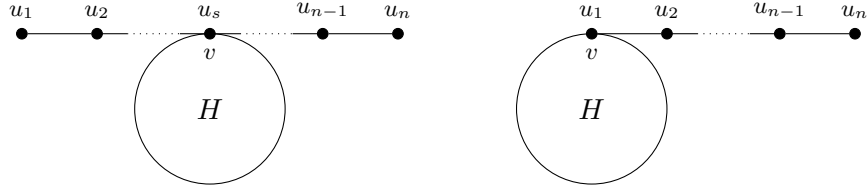
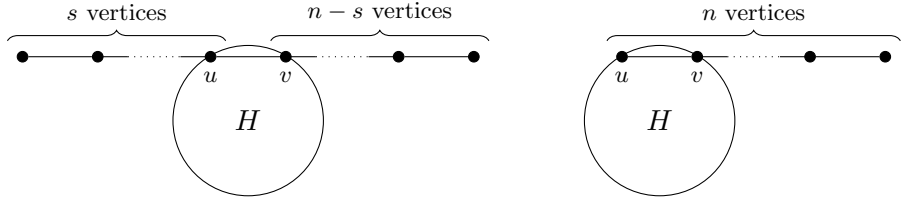
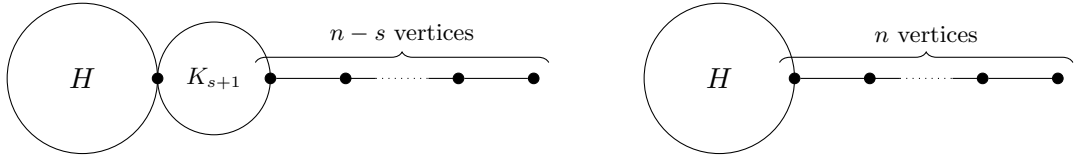

\medskip

We now give some terminology and some basic facts used in the proof of our main result.  Let $G$ be a  graph that has at least two blocks.  An \emph{end-block} is a block that contains exactly one cut vertex. A \emph{cyclic block} is a block that contains at least $3$ vertices. Thus, cyclic blocks in block graphs are complete graphs of order at least $3$. The \emph{block cut vertex tree} $T$ of a connected graph $G$ is the graph whose vertex set is the union of the set of blocks of $G$ and the set of cut vertices of $G$, where a block is joined to a cut vertex if and only if the cut vertex is contained in the block.
If $G$ is a block graph that is not a path, and if $G$ has a leaf $w$, then the shortest path $P$ from $w$ to a vertex $v$ of degree at least $3$ is called an {\em antenna} of $G$ {\em incident} with $v$. If $B$ is a block containing $v$ and no other vertex of $P$, we say that the antenna $P$ is {\em incident} with $B$.  Finally, we use the fact that $M_{P_n} = \frac{n+2}{3}$ (see~\cite{Jamison1983}), and the fact that $M_{K_n} = \frac{n2^{n-1}}{2^n -1}$ (see~\cite{KroekerMolOellermann2018}).  We are now ready to prove our main result.

\begin{theorem}
If $G$ is a connected block graph of order $n$, then $M_G\geq \frac{n+2}{3}$, with equality if and only if $G\cong P_n$.  In other words, the path $P_n$ is the unique connected block graph of order $n$ having minimum mean CIS order $\frac{n+2}{3}.$
\end{theorem}
\begin{proof}
Let $G\not\cong P_n$ be a connected block graph of order $n$.  We demonstrate that there is a connected block graph of order $n$ that has smaller mean CIS order than $G$, from which the statement follows.  Since the only connected block graphs of order $1$ and $2$ are paths, we may assume that $n \ge 3$. If $G$ has only one block, then $G \cong K_n$. Since $M_{P_n}=\frac{n+2}{3}< \tfrac{n2^{n-1}}{2^n-1}=M_G$ for $n \ge 3$, our claim follows.  Thus, we may assume that $G$ has at least two blocks.
For the remainder of the proof we consider two cases that depend on the structure of the block cut vertex tree $T$ of $G$.

If $T$ is a path, then it follows, since $G$ is not a path, that $G$ has a cyclic block. So either $G$ has a cyclic end-block or $G$ has a cyclic block that is incident with an antenna and contains exactly two cut-vertices. We can thus apply the Stretching Lemma to $G$ to obtain a connected block graph that has smaller mean CIS order.

Assume next that $T$ is not a path. Then $T$ contains a vertex of degree at least 3. So $T$ has a vertex  $u$ of degree at least 3, such that all but possibly one connected component of $T-u$ is isomorphic to a path. Let $\mathcal{C}$ be the family of components of $T-u$ that induce paths in $T-u$.  Suppose that at least one of the members of $\mathcal{C}$, $Q$ say,  contains a  block that is cyclic in $G$. In this case we can apply the Stretching Lemma to obtain a connected block graph with smaller mean CIS order than $G$. We may thus assume that all members of $\mathcal{C}$ correspond to antenna in $G$ that are incident with $u$.

Now $u$ is either a cut vertex of $G$, or a block of $G$. If $u$ is a cut vertex of $G$, then we can apply the Vertex Gluing Lemma to obtain a connected block graph whose mean CIS order is less than that of $G$.  So we may assume that $u$ is a block of $G$.  Let $v_1, v_2, \ldots, v_k$ be the cut vertices of $G$ that belong to the block $u$. By assumption, $k \ge 3$. Moreover, we may assume that, for $1 \le i \le k-1$, the component of $T-u$ containing $v_i$ is a path, i.e., belongs to $\mathcal{C}$.  By the argument of the previous paragraph, the members of $\mathcal{C}$ contain no cyclic blocks, and hence they correspond to antenna in $G$. Let $P_1$ and $P_2$ be the antenna of $G$ incident with $v_1$ and $v_2$, respectively. If we delete the vertices of $V(P_1) -\{v_1\}$ and $V(P_2)-\{v_2\}$ from $G$, neither $v_1$ nor $v_2$ is a cut vertex of the resulting graph.  By applying the Edge Gluing Lemma, we obtain a connected block graph with smaller mean CIS order than $G$.
\end{proof}

We close this section by describing a relationship between the mean sub-$k$-tree order of $k$-trees (see \cite{StephensOellermann2018}) and the mean CIS order of connected block graphs, and the implications of our main result in this setting.   The \emph{dual} $T'$ of a $k$-tree $T$ of order $n$ is defined to be the graph of order $n-k$ whose vertex set consists of the $(k+1)$-cliques in $T$, such that two vertices of $T'$ are adjacent if and only if the corresponding $(k+1)$-cliques in $T$ share a $k$-clique. It is not difficult to see that this dual of a nontrivial $k$-tree is a connected block graph. It was demonstrated in \cite{StephensOellermann2018} that there is a correspondence between the number of nontrivial sub-$k$-trees of a $k$-tree, and the number of connected induced subgraphs of its dual. Let $\mu(T)$ denote the mean order of the sub-$k$-trees of the $k$-tree $T$. Then the argument used in \cite{StephensOellermann2018}, for a certain type of $k$-tree, called a simple clique $k$-tree, can be used to show that for any $k$-tree $T$,
\[\mu(T) = \frac{W_{T'}}{N_{T'}+(n-k)k+1} +k.\]
 Since $T'$ has order $n-k$, it follows from our main result that $M_{T'} \ge M_{P_{n-k}} = \frac{n-k+2}{3}$. Moreover, $N_{T'} \ge N_{P_{n-k}}= \binom{n-k+1}{2}$. So, for fixed $k$, $\lim_{n\rightarrow \infty}\tfrac{\mu(T)}{M_{T'}} = 1$. In other words, $\mu(T) = M_{T'} + \mathrm{o}(n)$.  We conclude that understanding the behaviour of the mean CIS order of block graphs offers insight into the mean order of sub-$k$-trees of $k$-trees for large orders.

\section{The Local-Global Mean Inequality}\label{LocalGlobalSection}

In this section, we prove the local-global mean inequality for connected block graphs.  That is, we show that the local mean CIS order of a connected block graph $G$ at any vertex $v$ is greater than the (global) mean CIS order of $G$, extending the result of Jamison for trees~\cite[Theorem 3.9]{Jamison1983}.  We actually prove the stronger result that $M^*_{G,U}\geq M_G,$ where $U$ is either a single vertex, or the vertex set of a block of $G$.  In the process, we achieve several intermediate results which are used again later.

We begin by presenting a series of short lemmas that give inequalities between the number and/or total order of connected induced subgraphs of certain types in a given graph.  Some of these results hold for connected graphs in general.  First we note that \cite[Lemma 2.1]{WagnerWang2016} extends to the mean CIS order of connected graphs.  The proof is analogous to the proof of~\cite[Lemma 2.1]{WagnerWang2016}, and is omitted.

\begin{lemma}\label{WagLem}
Let $G$ be a connected graph and let $v$ be any vertex of $G$.  Then
$$W_{G, v}\leq\frac{N_{G, v}^2+N_{G, v}}{2},$$
with equality if and only if $G\cong P_n$ and $v$ is a leaf of $G$. \hfill\qed
\end{lemma}

We continue with three short lemmas whose proofs all use a similar technique.  Recall that $\mathcal{C}_{G}$ denotes the collection of connected induced subgraphs of $G$, and $\mathcal{C}_{G,v}$ denotes the collection of connected induced subgraphs of $G$ containing $v$. So $|\mathcal{C}_G|=N_G$ and $|\mathcal{C}_{G,v}|=N_{G,v}$.

\begin{lemma}\label{AdjacentVertexNumbers}
Let $G$ be a connected graph with adjacent vertices $u$ and $v$.  Then
\[
N_{G-v,u}\leq N_{G,v}-1,
\]
with equality if and only if $v$ is a leaf of $G$.
\end{lemma}

\begin{proof}
Define $f\colon\ \mathcal{C}_{G-v,u}\rightarrow \mathcal{C}_{G,v}$ by $f(H)=G[V(H)\cup\{v\}].$  One easily verifies that $f$ is well-defined and injective.  Further, the trivial graph on singleton vertex set $\{v\}$ is not in the image of $f$, hence
\[
N_{G-v,u}=|\mathcal{C}_{G-v,u}|\leq |\mathcal{C}_{G,v}|-1= N_{G,v}-1.
\]
If $v$ is a leaf of $G$ then every nontrivial connected induced subgraph of $G$ containing $v$ also contains $u$, and thus is in the image of $f$.  This gives equality.  Otherwise, $v$ has a neighbour $w\neq u$, and the nontrivial connected subgraph of $G$ induced by $\{v,w\}$ is not in the image of $f$, giving strict inequality.
\end{proof}

\begin{lemma}\label{NumCom}
Let $G$ be a connected block graph with non-cut vertex $v$.  Then
\[
N_{G, v}\leq N_{G-v}+1,
\]
with equality if and only if $G\cong K_n$.
\end{lemma}
\begin{proof}
Let $\mathcal{C}^+_{G,v}$ denote the collection of \emph{nontrivial} connected induced subgraphs of $G$ containing $v$.  Define a map $f\colon\ \mathcal{C}^+_{G,v}\rightarrow \mathcal{C}_{G-v}$ by $f(H)=H-v$. First we show that $f$ is well-defined.  Let $H\in \mathcal{C}^+_{G,v}.$  Since $\mathcal{C}^+_{G,v}$ contains only \emph{nontrivial} graphs, we note that $H-v$ has at least one vertex.  It remains to show that $H-v$ is connected.  Since $v$ is a non-cut vertex of $G$, and $G$ is a block graph, the open neighbourhood $N_H(v)$ induces a complete subgraph of $H$.  So the deletion of $v$ from $H$ does not separate any two neighbours of $v$.  Therefore, $H-v$ is connected and thus $f$ is well-defined.  Further, note that $f$ is injective, since if $f(H_1)=f(H_2)$, then $H_1$ and $H_2$ have the same vertex set.  Therefore,
\[
N_{G,v}=\left|\mathcal{C}^+_{G,v}\right|+1\leq \left|\mathcal{C}_{G-v}\right|+1=N_{G-v}+1.
\]
Equality is easily verified if $G\cong K_n$.  On the other hand, if $G\not\cong K_n$, then there is some vertex $u\in V(G)$ that is not adjacent to $v$.  Now $G[\{u\}]\in \mathcal{C}_{G-v},$ but $G[\{u,v\}]\not\in \mathcal{C}^*_{G,v},$ so we conclude that $f$ is not onto.  This gives strict inequality.
\end{proof}

\begin{lemma}\label{WeiNum}
Let $G$ be a connected graph of order $n \geq 2$ and let $u$ be any vertex of $G$.  Then
\[
N_G\leq W_{G, u}.
\]
\end{lemma}
\begin{proof}
Let
\[
\mathcal{D}_{G,u}=\{(X,x)\colon\ X\in\mathcal{C}_{G,u}, x\in V(X)\}.
\]
Observe that $|\mathcal{D}_{G,u}|=W_{G,u}$.  Therefore, it suffices to show that there is an injective function $f\colon\ \mathcal{C}_G\rightarrow \mathcal{D}_{G,u}$.
		
Assign a fixed ordering to the vertices of $G$.  For each vertex $v\in V(G)$, assign a fixed ordering to the shortest $u$--$v$ paths in $G$.  Among all vertices of $H$ closest to $u$, let $x_H$ be the vertex that appears first in the given ordering.
Let $P_H$ be the first path in the fixed ordering of shortest $u$--$x_H$ paths (if $H$ contains $u$ then $P_H$ consists of the single vertex $u$). Let $X_H$ be the subgraph of $G$ induced by $V(H) \cup V(P_H)$.  Clearly $X_H\in \mathcal{C}_{G,u}$. Define $f(H) = (X_H, x_H)$. Suppose $f(H_1)=f(H_2)=(X,x)$.  Let $P$ be the first shortest $u$--$x$ path in the fixed ordering of $u$--$x$ paths.  So $X=G[V(H_1)\cup V(P)]=G[V(H_2)\cup V(P)]$.  Note that no vertex of $P$ other than $x$ can lie in $H_1$ or $H_2$, since $P$ is a shortest path from $u$ to $H_1$, and from $u$ to $H_2$.  Therefore,
\[
H_1=X-(V(P)-x)=H_2.
\]
Thus, we have shown that $f$ is injective, and the desired conclusion follows.
\end{proof}

To prove the next lemma, we use an inductive argument and several lemmas already proven in this section.

\begin{lemma}\label{WeightBound}
Let $G$ be a connected block graph of order $n\geq 1$, and let $v$ be a vertex of $G$.  Then
\[
W_{G-v} \leq \frac{N_{G,v}N_{G-v}}{2},
\]
which is equivalent to
\[
M_{G-v}\leq \frac{N_{G,v}}{2}
\]
when $n\geq 2$.
\end{lemma}

\begin{proof}
We proceed by induction on $n$.  If $n=1$, then the statement holds trivially since $W_{G-v}=0$.  Now let $n\geq 2$ and suppose that the statement holds for all graphs of order less than $n$.  First of all, if $v$ is a cut vertex of $G$, then $M_{G-v}\leq M_H$, where $H$ is a component of $G-v$ of largest mean CIS order.  Let $H'=G[V(H)\cup \{v\}]$.  By the induction hypothesis applied to $H'$ at $v$, we have
\[
M_{G-v}\leq M_{H}=M_{H'-v}\leq \frac{N_{H',v}}{2}<\frac{N_{G,v}}{2}.
\]

So we may assume that $v$ is  not a cut vertex of $G$.  Let $u$ be a neighbour of $v$ in $G$, and write
\begin{align}\label{WeightDecomp}
W_{G-v}=W_{G-v,u}+W_{G-\{u,v\}}.
\end{align}
By Lemma \ref{WagLem}, we have
\[
W_{G-v,u}\leq \frac{N^2_{G-v,u}+N_{G-v,u}}{2},
\]
and by the induction hypothesis applied to $G-v$ at $u$, we have
\[
W_{G-\{u,v\}}\leq \frac{N_{G-v,u}N_{G-\{u,v\}}}{2}.
\]
Substituting these inequalities into~(\ref{WeightDecomp}), we obtain
\begin{align*}
W_{G-v} &\leq \frac{N^2_{G-v,u}+N_{G-v,u}}{2}+\frac{N_{G-v,u}N_{G-\{u,v\}}}{2}\\
&=\frac{N_{G-v,u}\left[N_{G-v,u}+1+N_{G-\{u,v\}}\right]}{2}\\
&=\frac{N_{G-v,u}\left[N_{G-v}+1\right]}{2}.
\end{align*}
Since $u$ and $v$ are adjacent in $G$, we may apply Lemma \ref{AdjacentVertexNumbers}, which gives
\begin{align*}
W_{G-v}&\leq \frac{\left[N_{G,v}-1\right]\left[N_{G-v}+1\right]}{2}\\
&=\frac{N_{G,v}N_{G-v}+N_{G,v}-N_{G-v}-1}{2}\\
&\leq \frac{N_{G,v}N_{G-v}}{2},
\end{align*}
as $N_{G, v}-N_{G-v}-1\leq 0$ by Lemma \ref{NumCom} (which applies since $v$ is not a cut vertex of $G$).
\end{proof}

For the next lemma, we introduce some new notation.  For a connected block graph $G$ of order at least $2$ with vertex $v$, define
\[
\mu_{G,v}=\frac{W_{G,v}-N_{G,v}}{N_{G,v}-1}.
\]
Note that $\mu_{G,v}$ is the average order of the connected induced subgraphs of $G-v$ containing at least one neighbour of $v$.  That is, $\mu_{G,v}=M^*_{G-v,N_G(v)}$.  In the next lemma, we show that if $v$ is not a cut vertex, then $\mu_{G,v}$ is at least as large as $M_{G-v}$.  We obtain our extension of the local-global mean inequality as a straightforward consequence of this result.

\begin{lemma}\label{MuLem}
Let $G$ be a connected block graph of order $n\geq 2$ with non-cut vertex $v$.
Then
\[
\mu_{G,v}\geq M_{G-v},
\]
with equality if and only if $G\cong K_n$.
\end{lemma}

\begin{proof}
We proceed by induction on the number of blocks $k$ of $G$.  If $k=1$, then $G\cong K_n$, and we verify that
\[
\mu_{G,v}=M_{G-v}.
\]
Now let $k>1$, and suppose that the statement holds for all connected block graphs of order at least $2$, with less than $k$ blocks.  Since $v$ is not a cut vertex, it is contained in only one block $B$ of $G$, and $N_G(v)=V(B)-v$.  Since we can partition the connected induced subgraphs of $G-v$ into all those that contain at least one neighbour of $v$ and all those that do not, we can write $M_{G-v}$ as a convex combination of $\mu_{G,v}$ and $M_{G-B}$.  So it suffices to show that $\mu_{G,v}\geq M_{G-B}.$
	
Let $U=V(B)-v=\{v_1,v_2,\dots, v_k\}$.  Note that $U$ is either a singleton, or induces a block in $G-v$.  For each $i\in\{1,\dots,k\}$, let $G_i$ be the connected component of $(G-v)-E(B)$ containing $v_i$.  By Lemma~\ref{LocalBlockLemma},
\[
\mu_{G,v}=M^*_{G-v,U}=\frac{N^*_{G-v,U}+1}{N^*_{G-v,U}}\sum_{i=1}^k \frac{W_{G_i,v_i}}{N_{G_i,v_i}+1}>\sum_{i=1}^k \frac{W_{G_i,v_i}}{N_{G_i,v_i}+1}.
\]
	
Now $G-B$ may be a disconnected graph, so $M_{G-B}$ is at most $M_H$, where $H$ is a connected component of $G-B$ of largest mean CIS order.  In particular, $H$ is a subgraph of $G_i-v_i$ for some $i\in\{1,\dots,k\}$.  Without loss of generality, suppose $H$ is a subgraph of $G_1-v_1$.  Let $H_1=G[V(H)\cup\{v_1\}]$.  By Lemma~\ref{LocalSum}, we have
\[
M_{G_1,v_1}\geq M_{H_1,v_1},
\]
and together with $W_{G_1,v_1}\geq W_{H_1,v_1},$ which clearly holds since every connected induced subgraph of $H_1$ containing $v_1$ is a connected induced subgraph of $G_1$ containing $v_1$, this implies
\[
\frac{W_{G_1,v_1}}{1+N_{G_1,v_1}}\geq \frac{W_{H_1,v_1}}{1+N_{H_1,v_1}}.
\]
Altogether, we have
\[
\mu_{G,v}>\frac{W_{G_1,v_1}}{1+N_{G_1,v_1}}\geq \frac{W_{H_1,v_1}}{1+N_{H_1,v_1}}=\frac{\mu_{H_1,v_1}(N_{H_1,v_1}-1)+N_{H_1,v_1}}{1+N_{H_1,v_1}}.
\]
Note that $H_1$ has order at least $2$, and fewer than $k$ blocks, and that $v_1$ is a non-cut vertex of $H_1$.  Hence, by the induction hypothesis, applied to $H_1$ at $v_1$,
\[
\mu_{G,v}> \frac{M_{H_1-v_1}(N_{H_1,v_1}-1)+N_{H_1,v_1}}{1+N_{H_1,v_1}}.
\]
Finally,
\[
\frac{M_{H_1-v_1}(N_{H_1,v_1}-1)+N_{H_1,v_1}}{1+N_{H_1,v_1}}\geq M_{H_1-v_1}
\]
is equivalent to
\[
W_{H_1-v_1}\leq \frac{N_{H_1,v_1}N_{H_1-v_1}}{2},
\]
which holds, by Lemma \ref{WeightBound}.  So we have shown that
\[
\mu_{G,v}> M_{H_1-v_1}=M_H\geq M_{G-B},
\]
and it follows, from our earlier observation, that $\mu_{G,v}> M_{G-v}.$
\end{proof}

\begin{theorem}[The Local-Global Mean Inequality]\label{LocalGlobal}
Let $G$ be a connected block graph.
\begin{enumerate}
\item \label{CorA} If $v$ is a vertex of $G$, then $M_{G}\leq M_{G,v}$.
\item \label{CorB} If $B$ is a block of $G$, then $M_{G}\leq M^*_{G,B}$.
\end{enumerate}
\end{theorem}

\begin{proof}
For \ref{CorA}, let $G'$ be the graph obtained from $G$ by adding a new leaf vertex $u$ to $v$.  Note that $\mu_{G',u}=M_{G,v}.$  Thus, by Lemma~\ref{MuLem},
\[
M_{G,v}=\mu_{G',u}\geq M_{G'-u}=M_{G}.
\]

For \ref{CorB}, let $G'$ be the graph obtained from $G$ by adding a new vertex $u$ and joining it to all vertices of $B$.  Note that $\mu_{G',u}=M^*_{G,B}.$  Thus by Lemma~\ref{MuLem},
\[
M^*_{G,B}=\mu_{G',u}\geq M_{G'-u}=M_G.  \qedhere
\]
\end{proof}

\section{The Vertex Gluing Lemma, the Edge Gluing Lemma, and the Stretching Lemma} \label{KeyLemmas}

In this section, we prove the Vertex Gluing Lemma, the Edge Gluing Lemma, and the Stretching Lemma.  The Vertex Gluing Lemma extends the Gluing Lemma of~\cite{MolOellermann2017} from trees to connected block graphs, and the proof is very similar once the local-global mean inequality is established for connected block graphs.  The proof is included in Appendix~\ref{Appendix} for completeness.

\begin{lemma}[The Vertex Gluing Lemma] \label{vertex_gluing}
Let $H$ be a connected block graph of order at least $2$ having vertex $v$.  Fix a natural number $n\geq 3$.  Let $P:u_1\dots u_n$ be a path of order $n$. For $s \in \{1, \dots, n\}$, let $G_s$ be the block graph obtained from the disjoint union of $P_n$ and $H$ by gluing $v$ to $u_s$.  If $1\leq i<j\leq \tfrac{n+1}{2}$, then $M_{G_i}<M_{G_j}.$  \hfill \qed
\end{lemma}

The proof of the Edge Gluing Lemma uses a similar technique. We use the fact that if $P_n$ is a path of order $n$, then $N_{P_n} = \binom{n+1 }{2}$ and $W_{P_n} = \binom{n+2}{3}$ (see  \cite{Jamison1983}).

\begin{lemma}[The Edge Gluing Lemma] \label{edge_gluing}
Let $H$ be a connected block graph of order at least $3$ with adjacent non-cut vertices $u$ and $v$.  Fix a natural number $n\geq 4$.  For $s\in\{1,2,\dots,n-1\}$, let $G_s$ be the graph obtained from $H\cup P_s\cup P_{n-s}$ by gluing a leaf of $P_s$ to $u$ and a leaf of $P_{n-s}$ to $v$.  If $1\leq i<j\leq \tfrac{n}{2}$, then $M_{G_i}<M_{G_j}.$
\end{lemma}

\begin{proof}
For ease of notation, let $u$ and $v$ also denote the glued vertices in $P_s$ and $P_{n-s}$, respectively.  Let $w$ be the vertex adjacent to $u$ in $P_s$ and let $z$ be the vertex adjacent to $v$ in $P_{n-s}$.  We have
\begin{align*}
\Phi_{G_s}(x)&=\Phi_{G_s,\{u,v\}}(x)+\Phi_{G_s-v,u}(x)+\Phi_{G_s-u,v}(x)+\Phi_{G_s-\{u,v\}}(x)\\
&=\Phi_{H,\{u,v\}}(x)\left(1+\Phi_{P_{s-1},w}(x)\right)\left(1+\Phi_{P_{n-s-1},z}(x)\right)+\Phi_{H-v,u}(x)\left(1+\Phi_{P_{s-1},w}(x)\right)\\
& \ \ \ \ +\Phi_{H-u,v}(x)\left(1+\Phi_{P_{n-s-1},z}(x)\right)+\Phi_{H-\{u,v\}}(x)+\Phi_{P_{s-1}}(x)+\Phi_{P_{n-s-1}}(x)\\
&=x\Phi_{H-v,u}(x)\left(1+\Phi_{P_{s-1},w}(x)\right)\left(1+\Phi_{P_{n-s-1},z}(x)\right)+\Phi_{H-v,u}(x)\left(1+\Phi_{P_{s-1},w}(x)\right)\\
& \ \ \ \ +\Phi_{H-v,u}(x)\left(1+\Phi_{P_{n-s-1},z}(x)\right)+\Phi_{H-\{u,v\}}(x)+\Phi_{P_{s-1}}(x)+\Phi_{P_{n-s-1}}(x)\\
&=\Phi_{H-v,u}(x)\left[x\left(1+\Phi_{P_{s-1},w}(x)\right)\left(1+\Phi_{P_{n-s-1},z}(x)\right)+\left(1+\Phi_{P_{s-1},w}(x)\right)+\left(1+\Phi_{P_{n-s-1},z}(x)\right)\right]\\
& \ \ \ \ +\Phi_{H-\{u,v\}}(x)+\Phi_{P_{s-1}}(x)+\Phi_{P_{n-s-1}}(x).
\end{align*}
Differentiating with respect to $x$, we obtain
\begin{align*}
\Phi'_{G_s}(x)&=\Phi'_{H-v,u}(x)\left[x\left(1+\Phi_{P_{s-1},w}(x)\right)\left(1+\Phi_{P_{n-s-1},z}(x)\right)+\left(1+\Phi_{P_{s-1},w}(x)\right)+\left(1+\Phi_{P_{n-s-1},z}(x)\right)\right]\\
&\ \ \ \ +\Phi_{H-v,u}(x)\bigg[\left(1+\Phi_{P_{s-1},w}(x)\right)\left(1+\Phi_{P_{n-s-1},z}(x)\right)+x\Phi'_{P_{s-1},w}(x)\left(1+\Phi_{P_{n-s-1},z}(x)\right)\\
& \hspace{3cm} +x\left(1+\Phi_{P_{s-1},w}(x)\right)\Phi'_{P_{n-s-1},z}(x)+\Phi'_{P_{s-1},w}(x)+\Phi'_{P_{n-s-1},z}(x)\bigg]\\
& \ \ \ \ +\Phi'_{H-\{u,v\}}(x)+\Phi'_{P_{s-1}}(x)+\Phi'_{P_{n-s-1}}(x).
\end{align*}
Evaluating at $x=1$ and letting $F=H-v$,
\begin{align*}
\Phi_{G_s}(1)=N_{F, u}\left(s(n-s)+n\right)+N_{F-u}+\tbinom{s}{2}+\tbinom{n-s}{2},
\end{align*}
and
\begin{align*}
\Phi'_{G_s}(1)&=W_{F, u}\left(s(n-s)+n\right)+N_{F, u}\left(s(n-s)+\tbinom{s}{2}(n-s+1)+(s+1)\tbinom{n-s}{2}\right)\\
& \ \ \ \ +W_{F-u}+\tbinom{s+1}{3}+\tbinom{n-s+1}{3}.
\end{align*}
By a straightforward computation,
\begin{align}\label{NGs}
\tfrac{d }{ds}N_{G_s}=\tfrac{d}{ds}\Phi_{G_s}(1)=N_{F, u}(n-2s)-(n-2s)=(n-2s)(N_{F, u}-1),
\end{align}
and
\begin{align}\label{WGs}
\tfrac{d}{ds}W_{G_s}=\tfrac{d}{ds}\Phi'_{G_s}(1)&=W_{F, u}(n-2s)+N_{F, u}\tfrac{n-2}{2}(n-2s)-\tfrac{n}{2}(n-2s)\nonumber \\
&=(n-2s)\left[W_{F, u}+\tfrac{n-2}{2}N_{F, u}-\tfrac{n}{2}\right]\nonumber \\
&=(n-2s)\left[W_{F, u}+\tfrac{n}{2}(N_{F, u}-1)-N_{F, u}\right].
\end{align}
Now we consider the mean $M_{G_s}=\tfrac{W_{G_s}}{N_{G_s}}$ as a continuous function of $s$. By the quotient rule, $\tfrac{d}{ds}M_{G_s}$ will have the same sign as $\tfrac{d}{ds}\left[W_{G_s}\right]N_{G_s}-W_{G_s}\tfrac{d}{ds}\left[N_{G_s}\right]$ on $\left[1, \frac{n}{2}\right)$, since the denominator of $\frac{d}{ds}M_{G_s}$ is strictly positive on $\left[1, \frac{n}{2}\right)$.  Since both $\tfrac{d}{ds}N_{G_s}$ and $\tfrac{d}{ds}W_{G_s}$ have a factor of $n-2s$, the expression
\begin{align}\label{SameSignAsDeriv}
\frac{\tfrac{d}{ds}\left[W_{G_s}\right]N_{G_s}-W_{G_s}\tfrac{d}{ds}\left[N_{G_s}\right]}{n-2s},
\end{align}
will have the same sign as $\frac{d}{ds}M_{G_s}$ for $s\in[1,n/2)$.  We show that (\ref{SameSignAsDeriv}) is strictly positive for $s\in[1,n/2)$.  Substituting (\ref{NGs}) and (\ref{WGs}) into (\ref{SameSignAsDeriv}), and then simplifying, we obtain
\begin{align}
&\left[W_{F, u}+\tfrac{n}{2}(N_{F, u}-1)-N_{F, u}\right]\left[N_{F, u}\left(s(n-s)+n\right)+N_{F-u}+\tbinom{s}{2}+\tbinom{n-s}{2}\right]\nonumber\\
& \ \ \  -(N_{F, u}-1)\bigg[W_{F, u}\left(s(n-s)+n\right)+N_{F, u}\left(s(n-s)+\tbinom{s}{2}(n-s+1)+(s+1)\tbinom{n-s}{2}\right)\nonumber\\
& \hspace{3cm} +W_{F-u}+\tbinom{s+1}{3}+\tbinom{n-s+1}{3}\bigg]\nonumber\\
&=W_{F,u}\tbinom{n+1}{2}-N_{F, u}\tfrac{n^2}{2}-\tfrac{n}{2}N_{F-u}+\tfrac{n-2}{2}N_{F, u}N_{F-u}-N_{F, u}^2\tfrac{n}{2}\nonumber\\
& \ \ \ \ +W_{F, u}N_{F-u}-W_{F-u}N_{F, u}+W_{F-u}+(N_{F, u}-1)\tfrac{n(n-1)(n-2)}{12}\nonumber \\
\begin{split}\label{eqwe}
&=\tfrac{n^2}{2}(W_{F, u}-N_{F, u})+\tfrac{n}{2}(W_{F, u}-N_{F,u}-N_{F-u})+\tfrac{n}{2}N_{F,u}(N_{F-u}+1-N_{F,u})\\
& \ \ \ \ +\left(W_{F, u}N_{F-u}-W_{F-u}N_{F, u}-N_{F, u}N_{F-u}+W_{F-u}\right)+(N_{F, u}-1)\tfrac{n(n-1)(n-2)}{12}.
\end{split}
\end{align}
We now explain why each term of (\ref{eqwe}) is nonnegative (and in fact two terms are positive).  The strict inequality $W_{F,u}-N_{F,u}>0$ is obvious. The inequality $W_{F,u}-N_{F,u}-N_{F-u}=W_{F,u}-N_F\geq 0$ follows by Lemma~\ref{WeiNum}, and the inequality $N_{F-u}+1-N_{F,u}\geq 0$ follows by Lemma~\ref{NumCom}.  The inequality $(N_{F, u}-1)\frac{n(n-1)(n-2)}{12}>0$ is immediate since $N_{F,u}\geq 2$ and $n\geq 4$.
Finally,
\begin{align}\label{MuIneq}
W_{F, u}N_{F-u}-N_{F, u}N_{F-u}-N_{F, u}W_{F-u}+W_{F-u}\geq 0 \ \ \ &\Longleftrightarrow \ \ \   \frac{W_{F,u}-N_{F,u}}{N_{F,u}-1}\geq \frac{W_{F-u}}{N_{F-u}}\nonumber \\
&\Longleftrightarrow \ \ \ \mu_{F,u}\geq M_{F-u},
\end{align}
and (\ref{MuIneq}) holds by Lemma \ref{MuLem}.  Therefore, the function $\frac{d}{ds}M_{G_s}$ is strictly positive on the interval $\left[1, \frac{n}{2}\right)$, and we conclude that $M_{G_1}<M_{G_2}<... <M_{G_{\lfloor n/2\rfloor}}$.
\end{proof}
	
Finally, we prove the Stretching Lemma.  We recall some notation used in the statement.  Fix an integer $n\geq 3$.  For $s\in\{1,2,\dots,n-1\}$, let $F_{s,n-s}$ (or $F_s$ for short) denote the graph obtained from the disjoint union of $K_{s}$ and $P_{n-s}$ by joining a leaf of $P_{n-s}$ to every vertex of $K_s$.  Note that if $s=1$, then $F_s\cong P_n$, while if $s=n-1$, then $F_s\cong K_n$.

Let $v$ be a vertex of $F_s$ belonging to the initial $K_s$.  By straightforward counting arguments,
\[
\Phi_{F_s,v}(x)=x(1+x)^{s-1}\sum_{i=0}^{n-s}x^i
\]
and
\begin{align*}
\Phi_{F_s-v}(x)
&=(1+x)^{s-1}\left(\sum_{i=0}^{n-s}x^i\right)+\Phi_{P_{n-s-1}}(x)-1.
\end{align*}

\begin{lemma}[The Stretching Lemma] \label{stretching_lemma}
Let $H$ be a connected block graph of order at least $2$ having a vertex $u$.  Fix a natural number $n\geq 3$.  Let $v$ be a vertex of $F_s=F_{s,n-s}$ belonging to the initial $K_s$.  For $s \in \{1, \dots, n-1\}$, let $G_s$ be the block graph obtained from the disjoint union of $H$ and $F_{s}$ by identifying $u$ and $v$.  If $1\leq i<j\leq n-1$, then $M_{G_i}<M_{G_j}.$
\end{lemma}

\begin{proof}
The general technique is to write $M_{G_s}$ as a function of $s$, treat $s$ as a real variable, and show that $M_{G_s}$ is increasing for $s\in[1,n-1]$.  In other words, we show that $\tfrac{d}{ds}\left[M_{G_s}\right]=\tfrac{d}{ds}\left[W_{G_s}/N_{G_s}\right]>0$ for all $s\in [1,n-1]$.  By the quotient rule, $\tfrac{d}{ds}\left[M_{G_s}\right]$ has the same sign as
\begin{align}\label{derivative}
\tfrac{d}{ds}\left[W_{G_s}\right]N_{G_s}-W_{G_s}\tfrac{d}{ds}\left[N_{G_s}\right],
\end{align}
so it suffices to show that this expression is positive for all $s\in[1,n-1]$.

Let $s\in[1,n-1]$.  We first derive expressions for $N_{G_s}$, $W_{G_s}$, $\tfrac{d}{ds}[N_{G_s}]$, and $\tfrac{d}{ds}[W_{G_s}]$.  We have
\begin{align*}
\Phi_{G_s}(x)&=\frac{\Phi_{F_s,v}(x)\Phi_{H,u}(x)}{x}+\Phi_{F_s-v}(x)+\Phi_{H-u}(x)\\
&=(1+x)^{s-1}\left(\sum_{i=0}^{n-s}x^i\right)\left[1+\Phi_{H,u}(x)\right]+\Phi_{H-u}(x)+\Phi_{P_{n-s-1}}(x)-1.
\end{align*}
Substituting $x=1$ yields
\begin{align}\label{NG}
N_{G_s}=2^{s-1}(n-s+1)(N_{H,u}+1)+N_{H-u}+\tbinom{n-s}{2}-1.
\end{align}
We also find
\begin{align*}
\Phi'_{G_s}(x)&=\left[(s-1)(1+x)^{s-2}\left(\sum_{i=0}^{n-s}x^i\right)+(1+x)^{s-1}\left(\sum_{i=0}^{n-s}ix^{i-1}\right)\right]\left[1+\Phi_{H,u}(x)\right]\\
& \ \ \ +(1+x)^{s-1}\left(\sum_{i=0}^{n-s}x^i\right)\left[\Phi'_{H,u}(x)\right]+\Phi'_{H-u}(x)+\Phi'_{P_{n-s-1}}(x).
\end{align*}
Substituting $x=1$ and simplifying yields
\begin{align}
W_{G_s}&=\left[(s-1)2^{s-2}(n-s+1)+2^{s-1}\tbinom{n-s+1}{2}\right]\left[N_{H,u}+1\right]\nonumber\\
& \ \ \ +2^{s-1}(n-s+1)W_{H,u}+W_{H-u}+\tbinom{n-s+1}{3}\nonumber \\
&=2^{s-2}(n-s+1)\left[(n-1)\left(N_{H,u}+1\right)+2W_{H,u}\right]+W_{H-u}+\tbinom{n-s+1}{3}.\label{WG}
\end{align}
Differentiating (\ref{NG}) and (\ref{WG}) with respect to $s$, and letting $L=\ln(2)$ for ease of reading, we find
\begin{align}
\tfrac{d}{ds}\left[N_{G_s}\right]&=2^{s-1}L\left(n-s+1-\tfrac{1}{L}\right)(N_{H,u}+1)-\tfrac{2(n-s)-1}{2}, \mbox{ and }\label{NGprime}\\
\tfrac{d}{ds}\left[W_{G_s}\right]&=2^{s-2}L\left(n-s+1-\tfrac{1}{L}\right)\left[\left(n-1\right)\left(N_{H,u}+1\right)+2W_{H,u}\right]-\tfrac{3(n-s)^2-1}{6}\label{WGprime}
\end{align}

For convenience, we let $t=n-s$ and rewrite (\ref{NG}), (\ref{WG}), (\ref{NGprime}), and (\ref{WGprime}) below.  Since $s\in[1,n-1]$, we have $t\in[1,n-1]$ as well.
\begin{align*}
N_{G_s}&=2^{s-1}\left(t+1\right)(N_{H,u}+1)+N_{H-u}+ \tbinom{t}{2}-1,\\
W_{G_s}&=2^{s-2}(t+1)\left[(n-1)\left(N_{H,u}+1\right)+2W_{H,u}\right]+W_{H-u}+\tbinom{t+1}{3},\\
\tfrac{d}{ds}[N_{G_s}]&=2^{s-1}L\left(t+1-\tfrac{1}{L}\right)(N_{H,u}+1)-\tfrac{2t-1}{2}, \mbox{ and }\\
\tfrac{d}{ds}[W_{G_s}]&=2^{s-2}L\left(t+1-\tfrac{1}{L}\right)\left[\left(n-1\right)\left(N_{H,u}+1\right)+2W_{H,u}\right]-\tfrac{3t^2-1}{6}.
\end{align*}
By substituting these expressions into (\ref{derivative}), expanding, and regrouping (and confirming with a computer algebra system), we find
\begin{align*}
\tfrac{d}{ds}\left[W_{G_s}\right]N_{G_s}-W_{G_s}\tfrac{d}{ds}\left[N_{G_s}\right]
&=\sum_{i=1}^5 E_i,
\end{align*}
where
\begin{align*}
E_1&=2^{s-2}\left(Lt+L-1\right)\left[(N_{H,u}+1)N_{H-u}-2W_{H-u}\right],\\
E_2&=2^{s-1}\left(Lt+L-1\right)\left[W_{H,u}N_{H-u}-N_{H,u}W_{H-u}\right],\\
E_3&=2^{s-3}(t+1)^2\left[(N_{H,u}+1)\left((n-1)(Lt-2L+1)-\tfrac{2}{3}(Lt^2+(2-L)t-1)\right)+2W_{H,u}(Lt-2L+1)\right],\\
E_4&=-\tfrac{1}{12}(t+1)^2(t^2-4t+2), \mbox{ and }\\
E_5&=2^{s-2}\left(Lt+L-1\right)(n-2)(N_{H,u}+1)N_{H-u}+\tfrac{2t-1}{2}W_{H-u}-\tfrac{3t^2-1}{6}N_{H-u}.
\end{align*}

We now show that $\sum_{i=1}^5 E_i>0$.  We first demonstrate that $E_1>0$ and $E_2\geq 0$.  Note that since $t\geq 1$, the factor $Lt+L-1\geq 2L-1>0$.  By Lemma~\ref{WeightBound}, $N_{H,u}N_{H-u}\geq 2W_{H-u}$.  Therefore,
\[
E_1\geq 2^{s-2}(Lt+L-1)N_{H-u}>0.
\]
By Theorem~\ref{LocalGlobal}\ref{CorA}, $M_{H,u}\geq M_{H-u}$, or equivalently $W_{H,u}N_{H-u}\geq W_{H-u}N_{H,u}$.  It follows immediately that $E_2\geq 0$.

Next we show that $E_3+E_4> 0$.  We begin by bounding $E_3$.  Since $H$ has order at least $2$, we have $W_{H,u}\geq N_{H,u}+1$.  This gives
\[
E_3\geq 2^{s-3}(t+1)^2(N_{H,u}+1)\left[(n+1)(Lt-2L+1)-\tfrac{2}{3}(Lt^2+(2-L)t-1)\right]
\]
Now we use the fact that $t\leq n-1$, or equivalently $n+1\geq t+2$.  This gives
\begin{align*}
E_3&\geq 2^{s-3}(t+1)^2(N_{H,u}+1)\left[(t+2)(Lt-2L+1)-\tfrac{2}{3}(Lt^2+(2-L)t-1)\right]\\
&=2^{s-3}(t+1)^2(N_{H,u}+1)\tfrac{1}{3}\left[Lt^2+(2L-1)t+8-12L\right]
\end{align*}
Finally, since $H$ has order at least $2$, we have $N_{H,u}\geq 2$.  Applying this inequality along with $s\geq 1$ gives
\[
E_3\geq \tfrac{1}{4}(t+1)^2\left[Lt^2+(2L-1)t+8-12L\right].
\]
Therefore,
\begin{align*}
E_3+E_4&\geq \tfrac{1}{4}(t+1)^2\left[Lt^2+(2L-1)t+8-12L\right]-\tfrac{(t+1)^2(t^2-4t+2)}{12}\\
&=\tfrac{1}{12}(t+1)^2\left[(3L-1)t^2+(6L+1)t+22-36L\right].
\end{align*}
Recalling that $L=\ln(2)$, one can verify that the quadratic in the square brackets of this last expression is positive for all $t\geq 1$.  Thus, we conclude that $E_3+E_4>0$.

Finally, we show that $E_5>0$.  We use the inequalities $W_{H-u}\geq N_{H-u}$, $N_{H,u}\geq 2$, $s\geq 1$, and $n-2\geq t-1$.
\begin{align*}
E_5&\geq \tfrac{3}{2}\left(Lt+L-1\right)(t-1)N_{H-u}+\tfrac{2t-1}{2}N_{H-u}-\tfrac{3t^2-1}{6}N_{H-u}\\
&=\tfrac{1}{6}\left[(9L-3)t^2-(9L+3)t+7\right]
\end{align*}
Recalling that $L=\ln(2)$, it is straightforward to verify that this last expression is strictly positive for all $t$, and hence $E_5>0$.

We have shown that $E_1>0,$ $E_2\geq 0$, $E_3+E_4>0$, and $E_5>0$ for all $s\in[1,n-1]$.  It follows that $\tfrac{d}{ds}\left[W_{G_s}\right]N_{G_s}-W_{G_s}\tfrac{d}{ds}\left[N_{G_s}\right]=\sum_{i=1}^5 E_i>0$ for all $s\in[1,n-1]$, as desired.
\end{proof}

\section{Concluding Remarks}

In this article, we demonstrated that among all connected block graphs of order $n$, the path has smallest mean CIS order.  This extends Jamison's result: {\em Among all trees of order $n$, the path has smallest mean subtree order}. Moreover, our main result lends support to the conjecture made in \cite{KroekerMolOellermann2018}:  {\em Among all connected graphs of order $n$, the path has minimum mean CIS order}.

\medskip

The problem of determining the structure of those block graphs, of a given order, with maximum mean CIS order remains open. It was conjectured by Jamison \cite{Jamison1983} that a tree with maximum mean subtree order among all trees of order $n$, called an {\em optimal} tree of order $n$, is a caterpillar. This is known as Jamison's Caterpillar Conjecture. This conjecture has been verified for all $n\leq 24$ (see~\cite{Jamison1983, MolOellermann2017}).  Mol and Oellermann~\cite{MolOellermann2017} made some progress on describing the structure of optimal trees. They proved that in any optimal tree of order $n$, every leaf is adjacent with a vertex of degree at least $3$, and that the number of leaves in an optimal tree of order $n$ is $\mathrm{O}(\log_2n)$ (moreover, the number of leaves is $\Theta(\log_2n)$ if Jamison's Caterpillar Conjecture is true).

Turning to block graphs, for $n\in\{3,4\}$, the complete graph has maximum mean CIS order among all block graphs of order $n$. We have verified that for $5\leq n\leq 11$, the block graph of order $n$ with maximum mean CIS order is a tree (more specifically, a caterpillar). We make the following conjecture, which strengthens Jamison's Caterpillar Conjecture.

\begin{conjecture}
For $n\geq 5$, if $G$ has maximum mean CIS order among all block graphs of order $n$, then $G$ is a caterpillar.
\end{conjecture}

\providecommand{\bysame}{\leavevmode\hbox to3em{\hrulefill}\thinspace}
\providecommand{\MR}{\relax\ifhmode\unskip\space\fi MR }
\providecommand{\MRhref}[2]{%
  \href{http://www.ams.org/mathscinet-getitem?mr=#1}{#2}
}
\providecommand{\href}[2]{#2}

\begin{appendix}

\appendix

\section{Proof of the Vertex Gluing Lemma}\label{Appendix}

\noindent
\textbf{The Vertex Gluing Lemma} (Lemma \ref{vertex_gluing})\\
\textit{Let $H$ be a connected block graph of order at least $2$ having vertex $v$.  Fix a natural number $n\geq 3$.  Let $P:u_1\dots u_n$ be a path of order $n$. For $s \in \{1, \dots, n\}$, let $G_s$ be the block graph obtained from the disjoint union of $P_n$ and $H$ by gluing $v$ to $u_s$.  If $1\leq i<j\leq \tfrac{n+1}{2}$, then $M_{G_i}<M_{G_j}.$}

\begin{proof} We may assume $s \le \frac{n+1}{2}$.  The connected induced subgraphs of $G_s$ can be partitioned into three types:
\begin{itemize}
\item Those that lie in $P_n$ but do not contain $u_s$.  These are counted by the polynomial $\Phi_{P_n}(x)-\Phi_{P_n, u_s}(x)$.
\item Those that lie in $H$ but do not contain $v$.  These are counted by the polynomial $\Phi_H(x)-\Phi_{H, v}(x)$.
\item Those that contain the glued vertex.  These are counted by the polynomial $\displaystyle\frac{\Phi_{P_n, u_s}(x)\Phi_{H, v}(x)}{x}$.
\end{itemize}
Thus,
\begin{align}\label{Phi}
\Phi_{G_s}(x)= \Phi_{P_n}(x)-\Phi_{P_n, u_s}(x)+\Phi_H(x)-\Phi_{H, v}(x) + \frac{\Phi_{P_n, u_s}(x)\Phi_{H, v}(x)}{x}.
\end{align}
Evaluating the derivative gives
\begin{align}\label{PhiPrime}
\begin{split}
\Phi'_{G_s}(x)&=  \Phi'_{P_n}(x)-\Phi'_{P_n, u_s}x)+\Phi'_H(x) -\Phi'_{H, v}(x)\\
& \ \ \ \ + \frac{\Phi'_{P_n, u_s}(x)\Phi_{H, v}(x)}{x}+\Phi_{P_n, u_s}(x)\left[\frac{x\Phi'_{H, v}(x)-\Phi_{H, v}(x)}{x^2}\right].
\end{split}
\end{align}
Evaluating (\ref{Phi}) and (\ref{PhiPrime}) at $x=1$ yields
\begin{align}\label{TsAt1}
\Phi_{G_s}(1)&= \Phi_{P_n}(1)+\Phi_H(1)-\Phi_{H, v}(1) + \Phi_{P_n, u_s}(1)\left[\Phi_{H, v}(1)-1\right],
\end{align}
and
\begin{align}\label{TsprimeAt1}
\begin{split}
\Phi'_{G_s}(1)&=  \Phi'_{P_n}(1)+\Phi'_H(1) -\Phi'_{H, v}(1)\\
& \ \ \ \ +\Phi'_{P_n, u_s}(1)\left[\Phi_{H, v}(1)-1\right]+\Phi_{P_n, u_s}(1)\left[\Phi'_{H, v}(1)-\Phi_{H, v}(1)\right],
\end{split}
\end{align}
respectively.


Note that $\Phi_{P_n}(1)= \tbinom{n +1}{2},$ $\Phi'_{P_n}(1) = \tbinom{n+2}{3},$  $\Phi_{P_n, u_s}(1) = s(n-s+1)$, and $\Phi'_{P_n, u_s}(1)=s(n-s+1)\tfrac{n+1}{2}.$  Using (\ref{TsAt1}) and (\ref{TsprimeAt1}) and substituting the values given in this paragraph, we obtain
\[
M_{G_s}=\frac{\Phi'_{G_s}(1)}{\Phi_{G_s}(1)}=\frac{\tbinom{n+2}{3}+ W_H-W_{H,v} + s(n-s+1)[(N_{H,v}-1)\tfrac{n+1}{2}+ W_{H,v} -N_{H,v}]}{\tbinom{n+1}{2} + N_H-N_{H,v} + s(n-s+1)[N_{H,v}-1]}.
\]

We show that if we view $M_{G_s}$ as a real valued function of $s \in \left[1, \frac{n+1}{2}\right]$, then $M_{G_s}$ is increasing on $\left[1, \frac{n+1}{2}\right]$.  Since the denominator of $M_{G_s}$ is strictly positive on the entire interval $\left[1, \frac{n+1}{2}\right]$, the derivative of $M_{G_s}$ exists and, by the quotient rule, it has the same sign as the function $f$ defined by
\begin{align*}
f(s)&=\tfrac{d}{ds}[\Phi'_{G_s}(1)]\Phi_{G_s}(1)-\tfrac{d}{ds}[\Phi_{G_s}(1)]\Phi'_{G_s}(1).
\end{align*}
Since $\tfrac{d}{ds}(s(n-s+1)) = n-2s+1$ is a factor of $f(s)$, we see that $f(s) = 0$ when $s = \frac{n+1}{2}$.  Moreover, for $s \in \left[1, \frac{n+1}{2}\right)$, we have $(n-2s+1)>0$ so that $f(s)$ has the same sign as
\begin{align*} \tfrac{f(s)}{n-2s+1}&= \left[(N_{H,v}-1) \tfrac{n+1}{2}+W_{H,v}-N_{H,v}\right]\left[\tbinom{n+1}{2} + N_H-N_{H,v} +s(n-s+1)(N_{H,v}-1)\right] \\
&\ \ \ \  - (N_{H,v}-1)\left[\tbinom{n+2}{3} + W_H - W_{H,v} + s(n-s+1)\left[(N_{H,v}-1)\tfrac{n+1}{2}+W_{H,v}-N_{H,v}\right]\right]\\
&=\left[(N_{H,v}-1)\tfrac{n+1}{2}+W_{H,v}-N_{H,v}][\tbinom{n+1}{2}+N_H-N_{H,v}\right]-(N_{H,v}-1)\left[\tbinom{n+2}{3} + W_H-W_{H,v}\right] \\
&=(N_{H,v}-1)\left[\tfrac{n+1}{2}\tbinom{n+1}{2} -\tbinom{n+2}{3}\right] +(W_{H,v}-N_{H,v})\tbinom{n+1}{2}+ (N_{H,v}-1)\tfrac{n+1}{2}(N_H-N_{H,v}) \\
&\ \ \ \ + (W_{H,v}-N_{H,v})(N_H-N_{H,v}) - (N_{H,v}-1)(W_H-W_{H,v})\\
&=(N_{H,v}-1)\left[\tfrac{n+1}{2}\tbinom{n+1}{2} -\tbinom{n+2}{3}\right] +(W_{H,v}-N_{H,v})\tbinom{n+1}{2}+ (N_{H,v}-1)\tfrac{n-1}{2}(N_H-N_{H,v}) \\
& \ \ \ \ +(N_{H,v}-1)(N_H-N_{H,v})+ (W_{H,v}-N_{H,v})(N_H-N_{H,v}) - (N_{H,v}-1)(W_H-W_{H,v})\\
&=(N_{H,v}-1)\tfrac{n(n+1)(n-1)}{12} + (W_{H,v}-N_{H,v})\tbinom{n+1}{2} +(N_{H,v}-1)\tfrac{n-1}{2}(N_H-N_{H,v})  \\
&\ \ \ \ +\left[(W_H - N_H) -(W_{H,v}-N_{H,v})\right]+ (N_HW_{H,v} -W_HN_{H,v}).
\end{align*}

Note that $\tfrac{f(s)}{n-2s+1}$ does not depend on $s.$  Thus, it suffices to show that each of the terms in the final expression for $\tfrac{f(s)}{n-2s+1}$ shown above is nonnegative (and at least one is strictly positive).  Indeed, using the straightforward inequalities $N_{H,v}>1$, $W_{H,v} > N_{H,v}$, $N_H > N_{H,v}$, and $n \ge 3$, it follows that
\begin{align*}
(N_{H,v}-1)\tfrac{n(n+1)(n-1)}{12}&>0,\\
(W_{H,v}-N_{H,v})\tbinom{n+1}{2}&>0, \mbox{ and}\\
(N_{H,v}-1)\tfrac{n-1}{2}(N_H-N_{H,v}) &>0.
\end{align*}
Let $k$ denote the order of $H$ and assume, for $1 \le i \le k$, that $H$ has $a_ i$ connected induced subgraphs of order $i$ and $b_i$ connected induced subgraphs of order $i$ that contain $v$. Then $a_i \ge b_i$ for $1 \le i \le k$ and $\displaystyle W_H=\sum_{i=1}^k ia_i$, $\displaystyle N_H =\sum_{i=1}^k a_i$, $\displaystyle W_{H,v} = \sum_{i=1}^k ib_i$ and $\displaystyle N_{H,v}=\sum_{i=1}^k b_i$.  Thus
\[
(W_H - N_H) -(W_{H,v}-N_{H,v})=\sum_{i=1}^k (i-1)(a_i-b_i)\ge 0.
\]
Finally, by Theorem~\ref{LocalGlobal}\ref{CorA},
\[
N_HW_{H,v} -W_HN_{H,v}=N_HN_{H,v}\left(\frac{W_{H,v}}{N_{H,v}}-\frac{W_H}{N_H}\right) > 0.
\]
We conclude that $f(s)$ is positive on $\left[1, \tfrac{n+1}{2}\right),$ so that $M_{G_s}$ is indeed increasing on $\left[1,\tfrac{n+1}{2}\right]$, and this completes the proof.
\end{proof}

\end{appendix}

\end{document}